\documentclass[12pt]{amsart}

\usepackage{mathrsfs,amssymb}

\newtheorem{theorem}{Theorem}[section]
\newtheorem{lemma}[theorem]{Lemma}

\newtheorem{cor}[theorem]{Corollary}

\theoremstyle{definition}

\theoremstyle{remark}
\newtheorem{remark}[theorem]{Remark}

\newtheorem{example}[theorem]{Example}

\numberwithin{equation}{section}

\DeclareMathOperator*{\esssup}{ess\,sup}

\let \la=\lambda

\let \d=\delta
\let \o=\omega
\let \a=\alpha
\let \f=\varphi
\let \b=\beta

\let \O=\Omega

\let \pa=\partial

\begin{document}

\title[pointwise sparse domination]
{Some remarks on the pointwise sparse domination}

\author{Andrei K. Lerner}
\address{Department of Mathematics,
Bar-Ilan University, 5290002 Ramat Gan, Israel}
\email{lernera@math.biu.ac.il}

\author{Sheldy Ombrosi}
\address{Departamento de Matem\'atica\\
Universidad Nacional del Sur\\
Bah\'ia Blanca, 8000, Argentina}\email{sombrosi@uns.edu.ar}

\thanks{A. Lerner is supported by ISF grant No. 447/16, S. Ombrosi is supported by CONICET PIP 11220130100329CO, Argentina.}

\begin{abstract}
We obtain an improved version of the pointwise sparse domination principle established by the first author  in \cite{Le2}.
This allows us to determine nearly minimal assumptions on a singular integral operator $T$ for which it admits a sparse domination.
\end{abstract}

\keywords{Sparse bounds, singular integrals, the $T1$ theorem.}
\subjclass[2010]{42B20, 42B25}
\maketitle

\section{Introduction}
Sparse bounds for different operators have been a recent and active topic in harmonic analysis. Their remarkable feature
is that an operator, which is typically signed and non-local, is dominated (pointwise or dually) by a positive and localized expression of the form
$$\sum_{Q\in {\mathcal S}}\langle f\rangle_{p,Q}\chi_Q(x),$$
where $\langle f\rangle_{p,Q}^p=\frac{1}{|Q|}\int_Q|f|^p$, and ${\mathcal S}$ is a sparse family of cubes of~${\mathbb R}^n$.

Recall that the family of cubes ${\mathcal S}$ is called $\eta$-sparse, $0<~\eta~\le~1$, if for every cube
$Q\in {\mathcal S}$, there exists a measurable set $E_Q\subset Q$ such that $|E_Q|\ge \eta|Q|$, and the sets
$\{E_Q\}_{Q\in {\mathcal S}}$ are pairwise disjoint.

Localization and sparseness are two main ingredients which make sparse bounds especially effective in quantitative weighted norm inequalities. 

The literature about sparse bounds is too extensive to be given here in more or less adequate form. We mention only
that sparse bounds for Calder\'on-Zygmund operators can be found in \cite{CR,HRT,La,Le1,Le2,LN}. Also, there are
several general sparse domination principles \cite{BM,BFP,CCPO,Le2}.

In \cite{Le2}, a sparse domination principle was obtained in terms of the grand maximal truncated operator
$${\mathcal M}_Tf(x)=\sup_{Q\ni x}\|T(f\chi_{{\mathbb R}^n\setminus 3Q})\|_{L^{\infty}(Q)}$$
defined for a given operator $T$.

\vskip 2mm
\noindent
{\bf Theorem A}\,\cite{Le2}.
{\it
Assume that $T$ is a sublinear operator of weak type $(q,q)$ and ${\mathcal M}_T$ is of weak type $(r,r)$, where $1\le q\le r<\infty$.
Then, for every compactly supported $f\in L^r({\mathbb R}^n)$, there exists a sparse family
${\mathcal S}$ such that
$$
|Tf(x)|\le C\sum_{Q\in {\mathcal S}}\langle f\rangle_{r,Q}\chi_Q(x)
$$
for a.e. $x\in {\mathbb R}^n$, where $C=c_{n,q,r}(\|T\|_{L^q\to L^{q,\infty}}+\|{\mathcal M}_T\|_{L^r\to L^{r,\infty}})$.
}
\vskip 2mm

In this paper we improve the above result by means of weakening the two main hypotheses.

First,
instead of the weak type $(q,q)$ of $T$ we assume a weaker property that there exists a non-increasing function $\psi_{T,q}(\la), 0<\la<1,$
such that for every cube $Q$ and for every $f\in L^q(Q)$,
$$
|\{x\in Q:|T(f\chi_Q)(x)|>\psi_{T,q}(\la)\langle f\rangle_{q,Q}\}|\le \la|Q|\quad(0<\la<1).
$$
We call this the $W_q$ property. If $T$ is of weak type $(q,q)$, then the $W_q$ property holds with $\psi_{T,q}(\la)=\|T\|_{L^q\to L^{q,\infty}}\la^{-1/q}$.

Second, we replace the operator ${\mathcal M}_T$
by a more flexible operator ${\mathcal M}^{\#}_{T,\a}$ defined for $\a>0$ by
$${\mathcal M}^{\#}_{T,\a}f(x)=\sup_{Q\ni x}\,\,\esssup_{x',x''\in Q}|T(f\chi_{{\mathbb R}^n\setminus \a Q})(x')-T(f\chi_{{\mathbb R}^n\setminus \a Q})(x'')|.$$

Our main result is the following.

\begin{theorem}\label{sdp} Let $T$ be a sublinear operator satisfying the $W_q$ condition and such that
${\mathcal M}^{\#}_{T,\a}$ is of weak type $(r,r)$ for some $\a\ge 3$, where $1\le q,r<\infty$. Let $s=\max(q,r)$.
Then, for every compactly supported $f\in L^s({\mathbb R}^n)$, there exists a $\frac{1}{2\cdot \a^n}$-sparse family ${\mathcal S}$ such that
$$
|Tf(x)|\le C\sum_{Q\in {\mathcal S}}\langle f\rangle_{s,Q}\chi_Q(x)
$$
for a.e. $x\in {\mathbb R}^n$, where $C=c_{n,r,s,\a}\big(\psi_{T,q}(1/12\cdot (2\a)^{n})+\|{\mathcal M}_T^{\#}\|_{L^r\to L^{r,\infty}}\big)$.
\end{theorem}

Theorem \ref{sdp} provides a more convenient tool compared to Theorem~A. Indeed,
there is no need now to work with the grand maximal truncated operator ${\mathcal M}_T$, which typically requires some additional effort.
In many particular cases (see examples in Section 5) the weak type $(r,r)$ of ${\mathcal M}^{\#}_{T,\a}$ follows from the estimate
$${\mathcal M}^{\#}_{T,\a}f(x)\le CM_rf(x),$$
where $M_rf=M(|f|^r)^{1/r}$ and $M$ is the Hardy-Littlewood maximal operator.

The fact that we do not require the weak type $(q,q)$ of $T$ in Theorem~\ref{sdp} allows us to obtain a sparse domination result for
a singular integral operator $T$ with minimal set of assumptions close in the spirit to the $T1$ theorem.

The paper is organized as follows. In Section 2 we present a proof of Theorem \ref{sdp}. We also show separately how this proof looks in the model case
of Calder\'on-Zygmund operators. In Section 3 we discuss some variations and extensions of Theorem \ref{sdp}. A sparse $T1$-type result is presented in Section 4.
Finally, in Section 5 we collect different examples (mostly known) of operators admitting the pointwise sparse domination. We show how Theorem \ref{sdp} simplifies
sparse bounds for these operators.

\section{Proof of Theorem \ref{sdp}}
In this section we prove Theorem \ref{sdp}. The proof is a variation of the proof of Theorem A, with an additional twist.
Although some parts of both proofs are almost identical, we provide a complete proof for reader's convenience.
First, we separate a common ingredient of both proofs in the following lemma.

\begin{lemma}\label{ci}
Assume that for any compactly supported $f\in L^s({\mathbb R}^n)$ and for every cube $Q$, there exists a
$\frac{1}{2}$-sparse family ${\mathcal F}_Q$ of subcubes of~$Q$ such that for a.e. $x\in Q$,
\begin{equation}\label{sploc}
|T(f\chi_{\a Q})|\chi_Q(x)\le C\sum_{R\in {\mathcal F}_Q}\langle f\rangle_{s,\a R}\chi_R(x)\quad(\a\ge 3).
\end{equation}
Then there exists a $\frac{1}{2\cdot \a^n}$-sparse family ${\mathcal S}$ such that for a.e. $x\in {\mathbb R}^n$,
\begin{equation}\label{sdom}
|Tf(x)|\le C\sum_{R\in {\mathcal S}}\langle f\rangle_{s,R}\chi_R(x).
\end{equation}
\end{lemma}

\begin{proof}
Take a partition of ${\mathbb R}^n$ by cubes $Q_j$ such that $\text{supp}\,(f)\subset \a Q_j$ for each $j$. For example, take a cube $Q_0$ such that
$\text{supp}\,(f)\subset Q_0$ and cover $3Q_0\setminus Q_0$ by $3^n-1$ congruent cubes $Q_j$. Each of them satisfies $Q_0\subset 3Q_j\subset \a Q_j$. Next, in the same way cover
$9Q_0\setminus 3Q_0$, and so on. The union of resulting cubes, including $Q_0$, will satisfy the desired property.

Having such a partition, apply (\ref{sploc}) to each $Q_j$ instead of $Q$. We obtain a $\frac{1}{2}$-sparse family ${\mathcal F}_{Q_j}$ such that for a.e. $x\in Q_j$,
$$
|Tf(x)|\le C\sum_{R\in {\mathcal F_{Q_j}}}\langle f\rangle_{s,\a R}\chi_R(x).
$$
Therefore, setting ${\mathcal F}=\cup_{j}{\mathcal F}_{Q_j}$, we obtain
that ${\mathcal F}$ is $\frac{1}{2}$-sparse and for a.e. $x\in {\mathbb R}^n$,
$$|Tf(x)|\le C\sum_{R\in {\mathcal F}}\langle f\rangle_{s,\a R}\chi_{R}(x).$$
This proves (\ref{sdom}) with a $\frac{1}{2\cdot \a^n}$-sparse family ${\mathcal S}=\{\a R: R\in {\mathcal F}\}$.
\end{proof}

Before giving the proof of Theorem \ref{sdp}, we show that it especially elementary in the model case of
Calder\'on-Zygmund operators.

We say that $T$ is a Calder\'on-Zygmund operator if $T$ is a linear operator of weak type $(1,1)$ such that
$$Tf(x)=\int_{{\mathbb R}^n}K(x,y)f(y)dy\quad\text{for all}\,\,x\not\in \text{supp}\,f$$
with kernel $K$ satisfying the smoothness condition
\begin{equation}\label{smooth}
|K(x,y)-K(x',y)|\le \o\left(\frac{|x-x'|}{|x-y|}\right)\frac{1}{|x-y|^n}
\end{equation}
for $|x-x'|<|x-y|/2$, where $[\o]_{\text{Dini}}=\int_0^1\o(t)\frac{dt}{t}<\infty.$

\begin{theorem}\label{czc} Let $T$ be a Calder\'on-Zygmund operator.
For every compactly supported $f\in L^1({\mathbb R}^n)$, there exists a $\frac{1}{2\cdot(5\sqrt n)^n}$-sparse family~${\mathcal S}$ such that for a.e. $x\in {\mathbb R}^n$,
\begin{equation}\label{sdcz}
|Tf(x)|\le C_n(\|T\|_{L^1\to L^{1,\infty}}+[\o]_{\rm{Dini}})\sum_{Q\in {\mathcal S}}\langle f\rangle_{1,Q}\chi_Q(x).
\end{equation}
\end{theorem}

This result is well known (see \cite{Le2} and the history therein). Its proof in \cite{Le2} is based on Theorem A.
The proof given below illustrates in a simplified form the main idea behind the proof of Theorem \ref{sdp}.

\begin{proof}[Proof of Theorem \ref{czc}]
Given a cube $Q$, denote $Q^*=5\sqrt nQ$. Let us also use the notation $f_Q=\frac{1}{|Q|}\int_Qf$.

It follows from the smoothness condition in the standard way that for every cube $P$ and for
all $x,x'\in P$,
\begin{equation}\label{pr1}
|T(f\chi_{{\mathbb R}^n\setminus P^*})(x)-T(f\chi_{{\mathbb R}^n\setminus P^*})(x')|\le c_n[\o]_{\text{Dini}}\inf_{P}Mf.
\end{equation}

Fix a cube $Q$. By the weak type $(1,1)$ of $M$ and $T$, there is $c_n'>0$ for which the set
$$\Omega=\Big\{x\in Q:\max\Big(\frac{M(f\chi_{Q^*})(x)}{c_n'},\frac{|T(f\chi_{Q^*})(x)|}{c_n'\|T\|_{L^1\to L^{1,\infty}}}\Big)>|f|_{Q^*}\Big\}$$
satisfies $|\Omega|\le \frac{1}{2^{n+2}}|Q|$. Denote $A=c_n'\|T\|_{L^1\to L^{1,\infty}}$.

Apply the local Calder\'on-Zygmund decomposition to $\chi_{\O}$ on $Q$ at height $\la=\frac{1}{2^{n+1}}$. We obtain a family of
pairwise disjoint cubes $\{P_j\}\subset~Q$ such that
\begin{equation}\label{covpr}
\frac{1}{2^{n+1}}|P_j|\le |P_j\cap\Omega|\le \frac{1}{2}|P_j|
\end{equation}
and $|\Omega\setminus\cup_jP_j|=0$. The latter property implies
\begin{equation}\label{TK}
|T(f\chi_{Q^*})(x)|\le A|f|_{Q^*}\quad\text{for a.e.}\,\,x\in Q\setminus \cup_jP_j.
\end{equation}

By (\ref{pr1}), for all $x\in P_j$ and $x'\in P_j\setminus \Omega$,
\begin{eqnarray}
|T(f\chi_{Q^*\setminus P_j^*})(x)|&\le&  c_n[\o]_{\text{Dini}}\inf_{P_j}M(f\chi_{Q^*})+|T(f\chi_{Q^*\setminus P_j^*})(x')|\nonumber\\
&\le& (c_nc_n'[\o]_{\text{Dini}}+A)|f|_{Q^*}+|T(f\chi_{P_j^*})(x')|.\label{pres}
\end{eqnarray}
Next, by (\ref{covpr}), $|P_j\setminus\Omega|\ge \frac{1}{2}|P_j|$. On the other hand,
$$|\{x\in P_j:|T(f\chi_{P_j^*})(x)|>A|f|_{P_j^*}\}|\le \frac{1}{2^{n+2}}|P_j|.$$
Therefore,
$$\inf_{P_j\setminus\Omega}|T(f\chi_{P_j^*})|\le A|f|_{P_j^*}\le c_n'A|f|_{Q^*},$$
which, combined with (\ref{pres}), implies that for all $x\in P_j$,
$$|T(f\chi_{Q^*\setminus P_j^*})(x)|\le A'|f|_{Q^*},$$
where $A'=\a_n(\|T\|_{L^1\to L^{1,\infty}}+[\o]_{\rm{Dini}})$.

From this and from (\ref{TK}), for a.e. $x\in Q$,
\begin{eqnarray}
|T(f\chi_{Q^*})|\chi_{Q}(x)&\le& A|f|_{Q^*}\chi_{Q\setminus\cup_jP_j}(x)+\sum_{j}|T(f\chi_{Q^*})|\chi_{P_j}(x)\nonumber\\
&\le &(A+A')|f|_{Q^*}+\sum_{j}|T(f\chi_{P_j^*})|\chi_{P_j}(x).\label{spc}
\end{eqnarray}

By (\ref{covpr}), $\sum_j|P_j|\le \frac{1}{2}|Q|$. Therefore, iterating (\ref{spc}), we obtain a $\frac{1}{2}$-sparse family ${\mathcal F}_Q$ of subcubes of $Q$ such that
$$
|T(f\chi_{Q^*})(x)|\le (A+A')\sum_{R\in {\mathcal F}_Q}|f|_{R^*}\chi_R(x)
$$
for a.e. $x\in Q$. It remains to apply Lemma \ref{ci}.
\end{proof}

The operator ${\mathcal M}_{T,\a}^{\#}$ in the above proof appears implicitly in (\ref{pr1}). In the proof of Theorem \ref{sdp}, it
appears explicitly and contributes to the exceptional set $\Omega$.

\begin{proof}[Proof of Theorem \ref{sdp}]
Given a cube $Q$, denote $Q^*=\a Q$. Next, set
$$\widetilde M_{T}f=\max(|Tf|,{\mathcal M}_{T,\a}^{\#}f).$$

By the weak type $(1,1)$ of $M$ and by the theorem assumptions along with H\"older's inequality, one can choose $c=c_{n,s,\a}>0$ and
$$A=2\psi_{T,q}(1/12\cdot (2\a)^{n})+c_{n,r,\a}\|{\mathcal M}_{T,\a}^{\#}\|_{L^r\to L^{r,\infty}}$$
for which the set
$$\Omega=\Big\{x\in Q:\max\Big(\frac{M_s(f\chi_{Q^*})(x)}{c},\frac{|\widetilde M_{T}(f\chi_{Q^*})(x)|}{A}\Big)>\langle f\rangle_{s,Q^*}\Big\}$$
satisfies $|\O|\le \frac{1}{2^{n+2}}|Q|$.

Apply the local Calder\'on-Zygmund decomposition to $\chi_{\O}$ on $Q$ at height $\la=\frac{1}{2^{n+1}}$. We obtain a family of
pairwise disjoint cubes $\{P_j\}\subset~Q$ such that
\begin{equation}\label{cov}
\frac{1}{2^{n+1}}|P_j|\le |P_j\cap\Omega|\le \frac{1}{2}|P_j|
\end{equation}
and $|\Omega\setminus\cup_jP_j|=0$. The latter property implies
\begin{equation}\label{Tle}
|T(f\chi_{Q^*})(x)|\le A\langle f\rangle_{s,Q^*}\quad\text{for a.e.}\,\,x\in Q\setminus \cup_jP_j.
\end{equation}

For almost all $x\in P_j$ and $x'\in P_j\setminus\O$,
\begin{eqnarray}
|T(f\chi_{Q^*\setminus P_j^*})(x)|&\le& \inf_{P_j}{\mathcal M}_{T,\a}^{\#}(f\chi_{Q^*})+|T(f\chi_{Q^*\setminus P_j^*})(x')|\nonumber\\
&\le& 2A\langle f\rangle_{s,Q^*}+|T(f\chi_{P_j^*})(x')|.\label{pres1}
\end{eqnarray}
Next, by (\ref{cov}), $|P_j\setminus\Omega|\ge \frac{1}{2}|P_j|$. On the other hand,
$$|\{x\in P_j:|T(f\chi_{P_j^*})(x)|>A\langle f\rangle_{s,P_j^*}\}|\le \frac{1}{2^{n+2}}|P_j|.$$
Therefore,
$$\inf_{P_j\setminus\Omega}|T(f\chi_{P_j^*})|\le A\langle f\rangle_{s,P_j^*}\le cA\langle f\rangle_{s,Q^*},$$
which, combined with (\ref{pres1}), implies that for all $x\in P_j$,
$$|T(f\chi_{Q^*\setminus P_j^*})(x)|\le (2+c)A\langle f\rangle_{s,Q^*}.$$

From this and from (\ref{Tle}), for a.e. $x\in Q$,
\begin{eqnarray}
|T(f\chi_{Q^*})|\chi_{Q}(x)&\le& A\langle f\rangle_{s,Q^*}\chi_{Q\setminus\cup_jP_j}(x)+\sum_{j}|T(f\chi_{Q^*})|\chi_{P_j}(x)\nonumber\\
&\le &(3+c)A\langle f\rangle_{s,Q^*}+\sum_{j}|T(f\chi_{P_j^*})|\chi_{P_j}(x).\label{sp}
\end{eqnarray}

By (\ref{cov}), $\sum_j|P_j|\le \frac{1}{2}|Q|$. Therefore, iterating (\ref{sp}), we obtain a $\frac{1}{2}$-sparse family ${\mathcal F}_Q$ of subcubes of $Q$ such that for a.e. $x\in Q$,
$$
|T(f\chi_{Q^*})(x)|\le (3+c)A\sum_{R\in {\mathcal F}_Q}\langle f\rangle_{s,R^*}\chi_R(x),
$$
which, along with Lemma \ref{ci}, completes the proof.
\end{proof}

\section{Some variations of Theorem \ref{sdp}}
We mention here some simple but useful variations/extensions of Theorem \ref{sdp}.
Let us start with the following, a slightly more precise version of Theorem \ref{sdp}.

\begin{theorem}\label{mpv} Let $1\le q,r<\infty$ and $s=\max(q,r)$. Let $f$ be a compactly supported function from $L^s({\mathbb R}^n)$.
Assume that $T$ is a sublinear operator satisfying the following property: there exist non-increasing functions $\psi$ and $\f$ such that for any cube $Q$,
$$
|\{x\in Q:|T(f\chi_Q)(x)|>\psi(\la)\langle f\rangle_{q,Q}\}|\le \la|Q|\quad(0<\la<1)
$$
and
$$
|\{x\in Q:{\mathcal M}_{T,\a}^{\#}(f\chi_Q)(x)>\f(\la)\langle f\rangle_{r,Q}\}|\le \la|Q|\quad(0<\la<1)
$$
for some $\a\ge 3$. Then there exists a $\frac{1}{2\cdot \a^n}$-sparse family ${\mathcal S}$ such that
$$
|Tf(x)|\le C\sum_{Q\in {\mathcal S}}\langle f\rangle_{s,Q}\chi_Q(x)
$$
for a.e. $x\in {\mathbb R}^n$, where $C=c_{n,s}\Big(\psi(1/12\cdot (2\a)^{n})+\f(1/12\cdot (2\a)^{n})\Big)$.
\end{theorem}

Indeed, the only difference in the proof is in the definition of $A$, namely, one should define
$$A=2\Big(\psi(1/12\cdot (2\a)^{n})+\f(1/12\cdot (2\a)^{n})\Big).$$
With this choice of $A$ we have
$$|\{x\in Q: \widetilde M_{T}(f\chi_{Q^*})(x)>A\langle f\rangle_{s,Q^*}\}|\le \frac{1}{6\cdot 2^n}|Q|,$$
and hence one can bound $|\Omega|$ by $\frac{1}{2^{n+2}}|Q|$.

\begin{remark}\label{adv}
The advantage of Theorem \ref{mpv} compared to Theorem \ref{sdp} is not only in the weaker assumption on
${\mathcal M}_{T,\a}^{\#}$ but also in the fact that the sparse domination for an individual function $f$ follows
from the initial assumptions on the same function. This advantage will be used in Theorem~\ref{t1v} below.
\end{remark}

Our next remark is that the $\langle f\rangle_{p,Q}$ averages in Theorems \ref{sdp} and \ref{mpv} can be replaced
by the Orlicz averages defined for a Young function $\Phi$ by
$$
\|f\|_{\Phi,Q}=\inf\Big\{\la>0:\frac{1}{|Q|}\int_Q\Phi(|f(y)|/\la)dy\le 1\Big\}.
$$
For example, the corresponding variant of Theorem \ref{mpv} can be stated as follows.

\begin{theorem}\label{ompv} Let $\Phi$ and $\Theta$ be Young functions such that $\Theta(t)\le C\Phi(t)$ for all $t\ge t_0\ge 0$.
Let $f$ be a compactly supported function from the Orlicz space $L^{\Phi}({\mathbb R}^n)$.
Assume that $T$ is a sublinear operator satisfying the following property: there exist non-increasing functions $\psi$ and $\f$ such that for any cube $Q$,
$$
|\{x\in Q:|T(f\chi_Q)(x)|>\psi(\la)\|f\|_{\Phi,Q}\}|\le \la|Q|\quad(0<\la<1)
$$
and
$$
|\{x\in Q:{\mathcal M}_{T,\a}^{\#}(f\chi_Q)(x)>\f(\la)\|f\|_{\Theta,Q}\}|\le \la|Q|\quad(0<\la<1)
$$
for some $\a\ge 3$. Then there exists a $\frac{1}{2\cdot \a^n}$-sparse family ${\mathcal S}$ such that
$$
|Tf(x)|\le C\sum_{Q\in {\mathcal S}}\|f\|_{\Phi,Q}\chi_Q(x)
$$
for a.e. $x\in {\mathbb R}^n$, where $C=c_{n,\Phi,\Theta}\Big(\psi(1/12\cdot (2\a)^{n})+\f(1/12\cdot (2\a)^{n})\Big)$.
\end{theorem}

Indeed, it is easy to see that $\langle f\rangle_{s,Q}$ appears in the sparse domination estimate in Theorem \ref{mpv} just because, by H\"older's inequality,
$$\max(\langle f\rangle_{q,Q}, \langle f\rangle_{r,Q})\le \langle f\rangle_{s,Q}.$$
Now, the assumption $\Theta(t)\le C\Phi(t)$ implies $\|f\|_{\Theta,Q}\le C\|f\|_{\Phi,Q}$. Therefore, replacing $\langle f\rangle_{s,Q}$ by $\|f\|_{\Phi,Q}$
in the proof of Theorems \ref{sdp}/\ref{mpv}, we obtain Theorem \ref{ompv}. For an application of Theorem \ref{ompv}, see Example~\ref{ex3} in Section 5.

We also note that Theorem \ref{mpv} can be easily extended to a multilinear case. In \cite{Li}, a multilinear extension of Theorem A was obtained.
Our multilinear variant of Theorem \ref{mpv} improves this result exactly in the same way as Theorem \ref{mpv} improves Theorem A.

Denote $\vec f=(f_1,\dots,f_m)$ and
$\vec f\chi_Q=(f_1\chi_Q,\dots,f_m\chi_Q)$. Given an operator $T(\vec f\,)$ and $\a>0$, define a multilinear analogue of the operator ${\mathcal M}_{T,\a}^{\#}$ by
$$
{\mathcal M}^{\#}_{T,\a}(\vec f\,)(x)
=\sup_{Q\ni x}\esssup_{x',x''\in Q}|\big(T(\vec f\,)-T(\vec f\chi_{\a Q})\big)(x')-\big(T(\vec f\,)-T(\vec f\chi_{\a Q})\big)(x'')|.
$$

\begin{theorem}\label{mmpv} Let $1\le q,r<\infty$ and $s=\max(q,r)$. Let $f_j, j=1,\dots,m,$ be compactly supported functions from $L^s({\mathbb R}^n)$,
and let $\vec f=(f_1,\dots, f_m)$.
Assume that $T$ is an operator satisfying the following property: there exist non-increasing functions $\psi$ and $\f$ such that for any cube $Q$,
$$
|\{x\in Q:|T(\vec f\chi_Q)(x)|>\psi(\la)\prod_{j=1}^m\langle f_j\rangle_{q,Q}\}|\le \la|Q|\quad(0<\la<1)
$$
and
$$
|\{x\in Q:{\mathcal M}_{T,\a}^{\#}(\vec f\chi_Q)(x)>\f(\la)\prod_{j=1}^m\langle f_j\rangle_{r,Q}\}|\le \la|Q|\quad(0<\la<1)
$$
for some $\a\ge 3$. Then there exists a $\frac{1}{2\cdot \a^n}$-sparse family ${\mathcal S}$ such that
$$
|T\vec f(x)|\le C\sum_{Q\in {\mathcal S}}\prod_{j=1}^m\langle f_j\rangle_{s,Q}\chi_Q(x)
$$
for a.e. $x\in {\mathbb R}^n$, where $C=c_{n,s}\Big(\psi(1/12\cdot (2\a)^{n})+\f(1/12\cdot (2\a)^{n})\Big)$.
\end{theorem}

We point out the necessary changes in the proof compared to the proof of Theorem \ref{mpv}. First, instead of $M_sf$ one should consider
$$M_s(\vec f\,)(x)=\sup_{Q\ni x}\prod_{j=1}^m\langle f_j\rangle_{s,Q}.$$
Second, $T(f\chi_{Q^*\setminus P_j^*})$ should be replaced by $T(\vec f\chi_{Q^*})-T(\vec f\chi_{P_j^*}).$
The rest of the proof is identically the same.

Note that in Theorem \ref{mmpv}, similarly to the corresponding result in~\cite{Li}, we do not assume that $T$ is multilinear (or multi(sub)linear), which is in contrast to the statement of its
linear analogue, Theorem~\ref{mpv}. The explanation is in the way we defined ${\mathcal M}_{T,\a}^{\#}$ in the linear and multilinear cases. The only place where the sublinearity of $T$ in Theorems \ref{sdp}
and \ref{mpv} was used is in the estimate
$$|T(f\chi_{Q^*\setminus P_j^*})|\le |T(f\chi_{Q^*})|+|T(f\chi_{P_j^*})|.$$
Having here $T(f\chi_{Q^*})-T(f\chi_{P_j^*})$ instead of $|T(f\chi_{Q^*\setminus P_j^*})|$, this estimate would hold trivially without any assumption on $T$.
Thus, defining ${\mathcal M}_{T,\a}^{\#}$ in the linear case in analogy with its multilinear analogue, one can state Theorems \ref{sdp} and \ref{mpv} for arbitrary $T$.

\section{A sparse $T1$-type theorem}
Consider a class of integral operators represented as
\begin{equation}\label{repr}
Tf(x)=\int_{{\mathbb R}^n}K(x,y)f(y)dy\quad\text{for all}\,\,x\not\in \text{supp}\,f.
\end{equation}

We say that $K$ satisfies the $L^r$-H\"ormander condition, $1\le r\le \infty$, if
\begin{eqnarray*}
\sup_Q\sup_{x,x'\in\frac{1}{2}Q}\sum_{k=1}^{\infty}|2^kQ|^{\frac{1}{r'}}
\|K(x,\cdot)-K(x',\cdot)\|_{L^{r}(2^kQ\setminus 2^{k-1}Q)}<\infty.
\end{eqnarray*}
Denote by $H_r$ the class of kernels satisfying the $L^r$-H\"ormander condition. It is easy to see that $H_1$ is just the classical H\"ormander condition, and
that $H_r\subset H_s$ if $r>s$.

Let $T^*$ denote the transpose of $T$, which is associated to the kernel $K^*(x,y)=K(y,x)$.
It is well known (see, e.g., \cite[p. 99]{D}) that if $T$ is $L^2$ bounded, represented by (\ref{repr}) for any $f\in L^2$ and if $K,K^*\in H_1$, then $T$ is of weak type $(1,1)$
and is bounded on $L^p$ for every $1<p<\infty$.

On the other hand, many results in the theory of singular integrals hold under stronger assumptions on $K$. Recall that $K$ is called standard kernel if it satisfies the size condition
$|K(x,y)|\le \frac{C}{|x-y|^n}$ for $x\not=y$ and both $K$ and $K^*$ satisfy the regularity condition
\begin{equation}\label{st}
|K(x,y)-K(x',y)|\le C\frac{|x-x'|^{\d}}{|x-y|^{n+\d}}\quad(|x-x'|<|x-y|/2).
\end{equation}

The $T1$ theorem \cite{DJ} in one of its equivalent forms asserts that if $K$ is standard, then $T$ is $L^2$ bounded if
and only if there exists $C>0$ such that for any cube $Q$,
\begin{equation}\label{t1}
\int_Q|T\chi_Q|dx\le C|Q|\quad\text{and}\quad \int_Q|T^*\chi_Q|dx\le C|Q|.
\end{equation}

In \cite{LM}, a ``sparse" proof of the $T1$ theorem was given, namely the sparse domination (in the dual form) for $T$ was obtained
assuming that $K$ is standard and $T$ satisfies (\ref{t1}).

It is still unknown what are the minimal regularity conditions on~$K$ for which the $T1$ theorem holds.
The sharpest known sufficient condition for the $T1$ theorem is (\ref{smooth}) for $K$ and $K^*$ with
$$\int_0^1\o(t)\Big(1+\log\frac{1}{t}\Big)^{1/2}\frac{dt}{t}<\infty$$
(see \cite{HT} for the corresponding discussion). In particular, it is unknown whether this condition can be relaxed to the classical Dini condition.

Similarly, one can ask about the minimal assumptions on $T$ yielding the pointwise sparse domination.
Our result in this direction is the following.

We assume that $K:{\mathbb R}^n\times {\mathbb R}^n\setminus \{(x,x):x\in {\mathbb R}^n\}$ is real valued, and that $T$ represented by (\ref{repr}) is properly defined on the space
$L^{\infty}_c$ of bounded functions with compact support.

\begin{theorem}\label{t1v}
Assume that $K\in H_r$ for some $1<r\le \infty$ and that $K^*\in H_1$. Suppose that there exists $C>0$ such that for every cube $Q$ and every measurable subset $E\subset Q$,
\begin{equation}\label{ncon}
\int_Q|T^*\chi_E|dx\le C|Q|.
\end{equation}
Then for every $f\in L^{\infty}_c$, there exists a $\frac{1}{2\cdot 3^n}$-sparse family ${\mathcal S}$ such that for a.e. $x\in {\mathbb R}^n$,
$$|Tf(x)|\le C(n,T)\sum_{Q\in {\mathcal S}}\langle f\rangle_{r',Q}\chi_Q(x).$$
\end{theorem}

We collect several standard facts. First, the assumption $K\in H_r$ along with H\"older's inequality implies
that for all $x\in {\mathbb R}^n$,
\begin{equation}\label{ff}
{\mathcal M}_{T,3}^{\#}f(x)\le CM_{r'}f(x).
\end{equation}
Second, the assumption $K^*\in H_1$ implies  that
for every cube $Q$ and any bounded function supported in $Q$ with $\int_Qf=0$,
\begin{equation}\label{at}
\int_{{\mathbb R}^n\setminus 2Q}|Tf|dx\le C\int_Q|f|dx.
\end{equation}

The proof of the following lemma is almost the same as the standard proof of the weak type $(1,1)$ of $T$.

\begin{lemma}\label{w1pr} Let $K^*\in H_1$. Assume that there exist $A,\d>0$ such that for every cube $Q$ and any $f\in L^{\infty}(Q)$,
\begin{equation}\label{con}
|\{x\in Q:|T(f\chi_Q)(x)|>\a\}|\le A\left(\frac{\|f\|_{L^{\infty}(Q)}}{\a}\right)^{\d}|Q|\quad(\a>0).
\end{equation}
Then there is $C>1$ such that for any $f\in L^{\infty}_c$ and for every cube $Q$,
$$
|\{x\in Q:|T(f\chi_Q)(x)|>\a|f|_Q\}|\le\frac{C}{\a^{\frac{\d}{1+\d}}}|Q|\quad(\a>0).
$$
\end{lemma}

\begin{proof}
Fix a cube $Q$. By homogeneity, one can assume that $|f|_Q=1$. Suppose also that $\a>1$ since otherwise the statement is trivial.

By the local Calder\'on-Zygmund decomposition, there exists a family of pairwise disjoint cubes $\{Q_j\}\subset Q$ such that
$$\a^{\frac{\d}{1+\d}}<|f|_{Q_j}\le 2^n\a^{\frac{\d}{1+\d}}$$
and $|f|\le \a^{\frac{\d}{1+\d}}$ a.e. on $Q\setminus \cup_jQ_j$.
Set $b=\sum_jb_j$, where $b_j=(f-f_{Q_j})\chi_{Q_j}$ and let $g=f-b$. Then $\|g\|_{L^{\infty}(Q)}\le 2^n\a^{\frac{\d}{1+\d}}$.

Applying (\ref{con}) yields
$$
|\{x\in Q:|T(g\chi_Q)(x)|>\a/2\}|\le \frac{2^{(n+1)\d}A}{\a^{\frac{\d}{1+\d}}}|Q|
$$

Next, by (\ref{at}),
$$\int_{{\mathbb R}^n\setminus 2Q_j}|T(b_j)|dx\le C\int_{Q_j}|b_j(y)|dy,$$
and, therefore,
\begin{eqnarray*}
\int_{{\mathbb R}^n\setminus\cup_j2Q_j}|T(b)|dx&\le&\sum_j\int_{{\mathbb R}^n\setminus 2Q_j}|T(b_j)|dx\\
&\le& C\sum_j\int_{Q_j}|b_j|dx\le 2C|Q|,
\end{eqnarray*}
which implies
\begin{eqnarray*}
&&|\{x\in Q:|T(b\chi_Q)(x)|>\a/2\}|\le |\cup_j2Q_j|\\
&&+\frac{2}{\a}\int_{{\mathbb R}^n\setminus\cup_j2Q_j}|T(b)|dx\le \left(\frac{2^n}{\a^{\frac{\d}{1+\d}}}+\frac{2C}{\a}\right)|Q|.
\end{eqnarray*}
Combining this with the above estimate for $T(g\chi_Q)$ yields
$$
|\{x\in Q:|T(f\chi_Q)(x)|>\a\}|\le \left(\frac{2^{(n+1)\d}A+2^n}{\a^{\frac{\d}{1+\d}}}+\frac{2C}{\a}\right)|Q|,
$$
which completes the proof.
\end{proof}

\begin{proof}[Proof of Theorem \ref{t1v}]
Condition (\ref{ncon}) implies that for every cube $Q$ and any $f\in L^{\infty}(Q)$,
$$\int_{Q}|T(f\chi_Q)|dx=\int_QfT^*\big((\text{sign}\,Tf)\chi_Q\big)dx\le 2C\|f\|_{L^{\infty}(Q)}|Q|.$$
Therefore, by Lemma \ref{w1pr}, there is $C'>1$ such that for any $f\in L^{\infty}_c$ and for every cube $Q$,
$$
|\{x\in Q:|T(f\chi_Q)(x)|>\a|f|_Q\}|\le\frac{C'}{\a^{1/2}}|Q|\quad(\a>0).
$$

From this and from (\ref{ff}), by the weak type $(r',r')$ of $M_{r'}$,
we obtain that both conditions of Theorem~\ref{mpv} are satisfied for $q=1$ and $r'$ instead of $r$ for every $f\in L^{\infty}_c$  with corresponding functions $\psi$ and $\f$ independent of $f$.
Applying Theorem \ref{mpv} completes the proof.
\end{proof}

Note that for every $\eta$-sparse family ${\mathcal S}$,
\begin{equation}\label{spbo}
\Big\|\sum_{Q\in {\mathcal S}}\langle f\rangle_{r,Q}\chi_Q\Big\|_{L^p}\le C_{n,p,r,\eta}\|f\|_{L^p}\quad(1\le r<p<\infty)
\end{equation}
(see, e.g., \cite[Lemma 4.5]{Le2}). This along with Theorem \ref{t1v} easily implies the following.

\begin{cor}\label{t1t}
Assume that $K\in H_r$ for some $1<r\le \infty$ and that $K^*\in H_1$. Then $T$ has a bounded extension that maps $L^2({\mathbb R}^n)$ to itself if and only if condition (\ref{ncon}) holds.
\end{cor}

\begin{proof}
The necessity part of this statement is obvious. Thus, we only need to show the sufficiency part.

By Theorem \ref{t1v} and by (\ref{spbo}), for any $p>r'$ and for all $f\in L^{\infty}_c$, 
\begin{equation}\label{ifp}
\|Tf\|_{L^p}\le c_{n,p,r}\|f\|_{L^p}.
\end{equation}
This along with (\ref{at}) implies the weak type $(1,1)$ property (restricted to $f\in L^{\infty}_c$)
$$\|Tf\|_{L^{1,\infty}}\le C\|f\|_{L^1}$$
(see, e.g., \cite{D} for this fact), and therefore, by interpolation, 
$$\|Tf\|_{L^2}\le C\|f\|_{L^2}\quad(f\in L^{\infty}_c),$$
which implies that $T$ can be extended continuously to a bounded mapping on~$L^2$.
\end{proof}

Observe that condition (\ref{ncon}) can be written in an equivalent and more symmetric form as follows: there exists $C>0$ such that for every cube $Q$ and any measurable subsets $E,F\subset Q$,
$$
|\langle T\chi_E,\chi_F\rangle|\le C|Q|.
$$

Having in mind Corollary \ref{t1t}, the question about the minimal regularity assumptions yielding the $T1$ theorem can be rephrased as follows:
what are the minimal regularity assumptions on $K$ for which the $T1$ conditions (\ref{t1}) imply (\ref{ncon})?

One can also ask whether the assumption $K\in H_r, r>1,$ in Corollary~\ref{t1t} can be further improved to the minimal assumption $K\in H_1$.

\section{Examples}
In this section we give several examples of operators admitting the pointwise sparse domination.
Note that most of the sparse bounds mentioned below are known. But here we provide a unified and
simplified approach to these results based on Theorem \ref{sdp} and its variants.
First, we mention the following corollary, which follows immediately from Theorem \ref{sdp}.

\begin{cor}\label{icor} Let $1\le q,r<\infty$.
Let $T$ be a sublinear operator of weak type $(q,q)$, and suppose that for some $\a\ge 3$ and for a.e. $x\in {\mathbb R}^n$,
$$
{\mathcal M}^{\#}_{T,\a}f(x)\le KM_rf(x).
$$
Let $s=\max(q,r)$.
Then, for every compactly supported $f\in L^s({\mathbb R}^n)$, there exists a $\frac{1}{2\cdot \a^n}$-sparse family ${\mathcal S}$ such that
\begin{equation}\label{sd}
|Tf(x)|\le C\sum_{Q\in {\mathcal S}}\langle f\rangle_{s,Q}\chi_Q(x)
\end{equation}
for a.e. $x\in {\mathbb R}^n$, where $C=c_{n,r,s,\a}(\|T\|_{L^q\to L^{q,\infty}}+K)$.
\end{cor}

\begin{example}\label{ex1}
Consider a class of integral operators represented by (\ref{repr}) with $K\in H_r, 1<r\le \infty$.
Then, as it was mentioned above, (\ref{ff}) holds.

Therefore, assuming additionally that $T$ is of weak type $(q,q)$,
by Corollary \ref{icor} we obtain that (\ref{sd}) holds with $s=\max(q,r')$. This result in a slightly different form can be found in~\cite{Li}.
\end{example}

\begin{example}\label{ex2}
Let ${\mathcal F}=\{\phi_{\a}\}_{\a\in A}$ be a family of real-valued measurable functions indexed by some set $A$, and let $T$ be an operator
represented by (\ref{repr}). Define the maximally modulated operator $T^{\mathcal F}$ by
$$T^{\mathcal F}f(x)=\sup_{\a\in A}|T({\mathcal M}^{\phi_{\a}}f)(x)|,$$
where ${\mathcal M}^{\phi_{\a}}f(x)={\rm e}^{2\pi i\phi_{\a}(x)}f(x)$.

Assume that $K\in H_r, 1<r\le \infty,$. Then (\ref{ff}) holds for $T^{\mathcal F}$ instead of $T$ with essentially the same proof.
Assuming additionally that $T^{\mathcal F}$ is of weak type $(q,q)$, we obtain (\ref{sd}) for $T^{\mathcal F}$ with $s=\max(q,r')$. The corresponding result can be found in \cite{B}.
\end{example}

\begin{example}\label{ex3} As a particular case of the previous example, consider the Carleson operator ${\mathcal C}$ defined by
$${\mathcal C}(f)(x)=\sup_{\xi\in {\mathbb R}}|H({\mathcal M}^{\xi}f)(x)|,$$
where $H$ is the Hilbert transform, and ${\mathcal M}^{\xi}f(x)={\rm e}^{2\pi i\xi x}f(x)$.

In this case $K\in H_{\infty}$, and therefore ${\mathcal M}^{\#}_{{\mathcal C},3}$ is of weak type $(1,1)$.

Set $\Phi(t)=t\log({\rm e}+t)\log\log\log({\rm e}^{{\rm e}^{\rm e}}+t).$ It was shown in \cite[Th.~5.1]{GMS} that for every interval $I\subset {\mathbb R}$,
$$\|{\mathcal C}(f\chi_I)\|_{L^{1,\infty}(I)}\le C|I|\|f\|_{\Phi,I}$$
(this represents an elaborated version of Antonov's theorem~\cite{A} on a.e. convergence of Fourier series for $f\in L\log L\log\log\log L$).

Therefore, by Theorem \ref{ompv}, for every compactly supported $f\in L^{\Phi}({\mathbb R})$, there exists a $\frac{1}{6}$-sparse family ${\mathcal S}$
such that for a.e. $x\in {\mathbb R}$,
$${\mathcal C}(f)(x)\le C\sum_{I\in {\mathcal S}}\|f\|_{\Phi,I}\chi_I(x).$$
\end{example}

\begin{example}\label{ex4}
Recall that a smooth function $a(x,\xi)$ defined on ${\mathbb R}^n\times {\mathbb R}^n$ belongs to the class $S_{\rho,\d}^m$ if
$$|\pa_x^{\a}\pa_{\xi}^{\b}a(x,\xi)|\le C_{\a,\b}(1+|\xi|)^{m-\rho|\b|+\d|\a|}$$
for all multi-indices $\a,\b$, where $m\in {\mathbb R}$ and $0\le \rho,\d\le 1$.

Given $a\in S_{\rho,\d}^m$, the pseudodifferential operator $T_a$ is defined by
$$T_af(x)=\int_{{\mathbb R}^n}a(x,\xi)\widehat f(\xi)e^{2\pi ix\cdot\xi}d\xi.$$

Assume that $a\in S_{\rho,\d}^{-n(1-\rho)}$, where $0<\rho\le 1$ and $0\le \d<1$.
First, $T_a$ is of weak type $(1,1)$ (this result can be found in \cite[Th. 3.2]{AH}). Second, for every $r>1$,
$$
{\mathcal M}^{\#}_{T_a,3}f(x)\le C_rM_rf(x)
$$
(for the proof of this estimate see \cite[Th. 3.3]{MRS}). Therefore, $T_a$ satisfies~(\ref{sd}) for all $s>1$.
This result was obtained in \cite{BC}.
\end{example}

\begin{example}\label{ex5}
Given a function $A$ with $\nabla A\in BMO$, define the operator $T_A$ by
$$T_Af(x)=\int_{{\mathbb R}^n}K(x,y)\frac{A(x)-A(y)-\nabla A(y)(x-y)}{|x-y|}f(y)dy\quad(x\not\in \text{supp}\,f),$$
where $K$ is standard kernel (as defined in Section 4).

An argument in \cite[Th. 1]{HY} shows that
$$
{\mathcal M}^{\#}_{T_A,5\sqrt n}f(x)\le CMMf(x).
$$
Therefore, assuming additionally that $T_A$ is of weak type $(q,q), q>1$, we obtain that $T_A$ satisfies (\ref{sd}) for $s=q$.
See \cite{H}, where a more refined result with a specific $K$ is obtained.
\end{example}

\end{document}